\documentclass[11pt]{amsart}\usepackage[left=3cm,right=3cm,top=3cm,bottom=3cm,a4paper]{geometry}   	 
\usepackage{mathtools}
\usepackage{amscd, amssymb, mathrsfs, mathabx, tikz-cd, amsthm, thmtools, stmaryrd}
\usepackage{kotex}       
\usepackage{todonotes}       
\input{xy}
\xyoption{all}

\usepackage{hyperref}
\hypersetup{
colorlinks   = true,          
urlcolor     = green,          
linkcolor    = blue,          
citecolor   = red          
}

\numberwithin{equation}{section}

\newcommand{\CC}{\mathbb{C}}

\newcommand{\PP}{\mathbb{P}}

\newcommand{\RR}{\mathbb{R}}

\newcommand{\ZZ}{\mathbb{Z}}

\newcommand{\cal}{\mathcal}

\newcommand\cA{{\cal A}}

\newcommand\cE{{\cal E}}
\newcommand\cF{{\cal F}}
\newcommand\cG{{\cal G}}
\newcommand\cH{{\cal H}}
\newcommand\cI{{\cal I}}

\newcommand\cQ{{\cal Q}}

\newcommand\cW{{\cal W}}

\def\fE{\mathfrak{E}}
\def\fF{\mathfrak{F}}

\def\fP{\mathfrak{P}}

\def\fZ{\mathfrak{Z}}

 \DeclareMathOperator{\Hom}{Hom}
\DeclareMathOperator{\End}{End}  
 
\DeclareMathOperator{\id}{id} 
 
\DeclareMathOperator{\tr}{tr} 

\DeclareMathOperator{\rank}{rank}

\DeclareMathOperator{\ad}{ad}

\def\vphi{\varphi}

\def\surra{\twoheadrightarrow}

\def\hra{\hookrightarrow}

\def\lra{\longrightarrow }
\def\bar{\overline}
\def\rou{\partial}

\def\wtil{\widetilde}

\def\ker{\mathrm{ker}}

\def\deg{\mathrm{deg}}
\def\dim{\mathrm{dim}}
\def\codim{\mathrm{codim}}

\def\rank{\mathrm{rank}}

\def\Vol{\mathrm{Vol}}
\def\Higgs{\mathrm{Higgs}}
\def\fg{\mathfrak{g}}
\def\fz{\mathfrak{z}}
\def\fp{\mathfrak{p}}
\def\fh{\mathfrak{h}}

\newtheorem{prop}{Proposition}[section]
\newtheorem{theo}[prop]{Theorem}
\newtheorem{lemm}[prop]{Lemma}
\newtheorem{coro}[prop]{Corollary}
\newtheorem{rema}[prop]{Remark}

\newtheorem{defi}[prop]{Definition}

\def\dual{^{\vee}}

\def\Sym{\mathrm{Sym}}

\def\Quot{\mathrm{Quot}}

\def\Gr{\mathrm{Gr}}

\def\SL{\mathrm{SL}}

\def\Ad{\mathrm{Ad}}

\def\Bun{\mathrm{Bun}}

\newcommand{\set}[1]{\{ #1 \}}

\newcommand{\bracket}[1]{\langle #1 \rangle}

\title[]{Infinitesimal automorphisms and obstruction theory on the moduli of $L$-valued $G$-Higgs bundles}

\author{Sanghyeon Lee}
\address{Ajou University,
206 World cup-ro, Suwon, Republic of Korea}
\email{sanghyeon25@ajou.ac.kr}

\author{Sang-Bum Yoo}
\address{Department of Mathematics Education, Gongju National University of Education, 27 Ungjin-ro, Gongju-si, Chungcheongnam-do, 32553, Republic of Korea}
\email{sbyoo@gjue.ac.kr}

\begin{document}

\maketitle

\begin{abstract}
For an arbitrary reductive group $G$, we compute the infinitesimal automorphisms of $L$-valued principal $G$-Higgs bundles over a compact K\"ahler manifold $X$, extending known results for $\Omega_X^{1}$-valued $G$-Higgs bundles. 

Using this computation, when $G$ is semisimple and $X$ is a smooth projective variety, we show that the moduli stack of stable $L$-valued $G$-Higgs bundles is a Deligne-Mumford (DM) stack. 

Furthermore, when $X$ is a smooth projective surface and $L=K_X$, we construct a symmetric perfect obstruction theory on this stable locus. We expect this will provide a foundation for defining Vafa-Witten invariants for reductive groups $G$.
\end{abstract}

\setcounter{tocdepth}{2}

\tableofcontents

\section{Introduction}

Let $X$ be a compact connected smooth K\"ahler manifold, equipped with a fixed K\"ahler form $\omega$. For a reductive linear algebraic group $G$, the Hitchin-Kobayashi type correspondence for the stability of principal $G$-bundles on $X$, and $G$-bundles with additional data has been studied for a long time.

In \cite{AB01}, B. Anchouche and I. Biswas proved that a holomorphic principal $G$-bundle $E_G$ has an Einstein-Hermitian connection if and only if $E_G$ is polystable. (See \cite[Definition 3.5]{AB01} for the definition of polystability.) As an intermediate step, they also proved that if $E_G$ is stable, then the adjoint vector bundle $\ad(E_G):= E_G \times_G \fg$ is polystable where $\fg$ is the Lie algebra of the Lie group $G$. 

Furthermore, in \cite{BS09}, I. Biswas and G. Schumacher proved that an $\Omega_X^{1}$-valued $G$-Higgs sheaf admits a hermitian Yang-Mills-Higgs connection if and only if it is polystable. \footnote{$\Omega_X^{1}$ is a sheaf of holomorphic 1-forms over $X$.} 
Moreover, in \cite{Bis09}, I. Biswas proved that the infinitesimal automorphisms of an $\Omega_X^{1}$-valued stable $G$-Higgs sheaves is isomorphic to the center of the Lie algebra, $\fz(\fg)$.

In Sections \ref{sect:stability} and \ref{sect:infaut} of this paper, we will extend the computation of infinitesimal automorphisms to $L$-valued $G$-Higgs bundles, $(E_G,\vphi)$ where $L$ is a vector bundle on $X$ and $\vphi \in H^0(X,\ad(E_G)\otimes L)$. For that, we use a generalized Kobayashi-Hitchin correspondence in \cite{ACGP03}. For a quiver $Q$, L. Álvarez-Cónsul and O. García-Prada proved that a holomorphic twisted quiver bundle is polystable if and only if it admits a hermitian metric satisfying a quiver vortex equation. Note that a twisted quiver bundle corresponding to a quiver $Q$ is a collection of vector bundles attached to each vertex of the quiver $Q$, and the data of morphisms of vector bundles with some twistings attached to each arrow of $Q$. See \cite[Definition 1.1]{ACGP03} for details of the twisted quiver bundles. When $Q$ is a quiver with a single vertex and a single arrow, and the twisting is given by $(\Omega_X^{1})\dual$, \cite[Theorem 3.1]{ACGP03} recovers the Kobayashi-Hitchin correspondence for $\Omega_X^{1}$-valued Higgs bundles in \cite{Hit87, Sim88}. 

As an application of the computation of the infinitesimal automorphisms to the case where $X$ is a smooth projective variety, in Section \ref{sect:POT}, we will prove that the moduli space of $L$-valued stable $G$-Higgs bundles, denoted by $\Higgs_L^s(X)$ is a Deligne-Mumford (DM) stack. Moreover, when $\dim(X)=2$ and $L=K_X$, we will prove that the moduli $\Higgs_{K_X}^s (X)$ is equipped with a symmetric perfect obstruction theory. We expect that this study may provide a foundation to define Vafa-Witten invariants for a general reductive group $G$.

Note that in the case when $G$ is a symplectic group or special orthogonal group and $L=K_X$, a perfect obstruction theory is constructed in \cite{Sch25}. They considered the moduli space of symplectic Higgs bundles and orthogonal Higgs bundles as the fixed loci of the moduli space of Higgs bundles with fixed determinant(say $\Higgs_\SL$), via the $\ZZ_2$-action acting on the Higgs fields by $\vphi \mapsto \vphi\dual$. They proved that the derived structure of $\Higgs_\SL$ induces a derived structure on the $\ZZ_2$-fixed locus. 

In the future, we would like to generalize the results of this paper to Gieseker stable $G$-Higgs sheaves. When we consider a minimal second Chern class $c_{2, \min} \in \ZZ$, similarly to the symplectic and orthogonal cases in \cite{Sch25}, the moduli space of $K_X$-valued $G$-Higgs bundles with the minimal second Chern class 
has a compact $\CC^*$-fixed locus, where the action is given by scaling of the Higgs fields $\vphi$. Then, we can define Vafa-Witten invariants of the group $G$ by using the equivariant virtual fundamental class in the same manner as in \cite{TT19}.


\subsection*{Acknowledgements} We especially thank Yaoxiong Wen for the long-standing discussions regarding Higgs bundles and their moduli. We thank Georg Schumacher for explaining the details in the paper \cite{BS09}. We also thank Yuuji Tanaka for explaining the generalized Kobayashi-Hitchin correspondence in \cite{ACGP03} in detail. 

Sanghyeon Lee was supported by the National Research Foundation of Korea(NRF) grant funded by the Korea government (MSIT) (RS-2026-25483737).

\section{Stability condition and Yang-Mills-Higgs equation for principal $L$-valued $G$-Higgs bundles.} \label{sect:stability}
Let us denote by $\Lambda_{\omega}$ the contraction operator with respect to the K\"ahler form $\omega$. 
Let $G$ be a connected reductive linear algebraic group defined over $\mathbb{C}$ with Lie algebra $\mathfrak{g}$. 
An \textbf{$L$-valued principal $G$-Higgs bundle} on $X$ is a pair $(E_G,\vphi)$ where $E_G$ is a principal $G$-bundle on $X$ and $\vphi$ is a section:
\[
\vphi \in H^0(X, \ad(E_G) \otimes L)
\]
where $L$ is a holomorphic vector bundle on $X$ equipped with a Hermitian metric. We call the section $\vphi$ an ($L$-valued) Higgs field. 
Following \cite{Bis09}, we define the stability of $L$-valued $G$-Higgs bundle as follows. In this paper, we call an open set $U \subset X$ a \textbf{big open subset} if $\codim(X\setminus U) \ge 2$.

On the other hand, when we use the term $L$-valued Higgs (not $G$-Higgs) bundle, it means a pair $(E,\vphi)$ where $E$ is a vector bundle and $\vphi \in \Hom(E, E\otimes L)$.

\begin{defi}\label{def:stability of VW G-bundle}
The $L$-valued $G$-Higgs bundle $(E_G,\vphi)$ on $X$ is called \textbf{semistable} (resp. \textbf{stable}) if for any maximal parabolic subgroup $P \subset G$ and its holomorphic reduction $E_P \subset E_G|_{U}$ over some big open subset $U$ such that the Higgs field $\vphi|_{U} \in H^0(U, \ad(E_G) \otimes L)$ restricts to $\vphi|_{U} \in H^0(U, \ad(E_P)\otimes L)$, we always have the inequality:
\[
\deg(\ad(E_G|_{U})) \ge \deg(\ad(E_P)) \,(\text{resp. }\deg(\ad(E_G|_{U})) > \deg(\ad(E_P))) \, .
\]
\end{defi}


\begin{rema}
In \cite{AB01}, the definition of stability is slightly different as follows. In \cite{AB01}, the authors considered the $L=\Omega_X^{1}$ case. A principal ($\Omega_X^{1}$-valued) $G$-Higgs bundle $(E_{G},\vphi)$ on $X$ is called \textbf{semistable} (resp. \textbf{stable}) if for any maximal parabolic subgroup $P\subset G$ and for any reduction of the structure group $\sigma:U\to E_{G}/P := E_G \times_G (G/P)$ over some open subset $U$ with $\codim(X\setminus U)\ge2$ satisfying $\vphi\in H^{0}(X, \ad(E_P) \otimes \Omega_X^{1})$, we have the inequality $\deg\,\sigma^{*}(T_{E_{G}/P})\ge0$ (resp. $\deg\,\sigma^{*}(T_{E_{G}/P})>0$), where $T_{E_{G}/P}$ is the relative tangent bundle for the natural projection $E_{G}/P|_{U}\to U$. We can see that this definition is equivalent to Definition \ref{def:stability of VW G-bundle}. Consider the short exact sequence of $P$-modules (the module structures are given by the adjoint representation)
\[
0 \to \fp \to \fg \to \fg/\fp \to 0,
\]
where $\fg$ and $\fp$ are Lie algebras of $G$ and $P$ respectively. It induces the short exact sequence of vector bundles on $U$
\[
0 \to \sigma^{*}(E_{G}\times_{P}\fp) \to \sigma^{*}(E_{G}\times_{P}\fg) \to \sigma^{*}(E_{G}\times_{P}(\fg/\fp)) \to 0.
\]
This is nothing but
\[
0\to\ad(E_{P})\to\ad(E_{G})|_{U}\to\sigma^{*}(T_{E_{G}/P})\to0.
\]
Hence $\deg\,\sigma^{*}(T_{E_{G}/P})\ge0$ (resp. $\deg\,\sigma^{*}(T_{E_{G}/P})>0$) if and only if $\deg(\ad(E_G|_{U})) \ge \deg(\ad(E_P))$ (resp. $\deg(\ad(E_G|_{U})) > \deg(\ad(E_P))$).
\end{rema}


For a holomorphic vector bundle $\cE$ equipped with a Hermitian metric $h_\cE$, there exists a unique connection $D_\cE$, called the \textbf{Chern connection}, which satisfies two conditions:
\begin{enumerate}
    \item It is compatible with the holomorphic structure: $D_\cE^{0,1} = \bar{\partial}_\cE$. Here, $\bar{\partial}_\cE$ is the Dolbeault operator defining the holomorphic structure of $\cE$.
    \item It preserves the metric $h$: $dh(s, t) = h(D_\cE s, t) + h(s, D_\cE t)$ for sections $s, t \in H^0(X,\cE)$.
\end{enumerate}
We denote the curvature of the Chern connection by $F_\cE = D_\cE^2$.

For an admissible Hermitian metric $h_\fE$ on $\fE := \ad(E_G)$, the quiver vortex equation is given by \cite[Definition 2.1, Theorem 3.1]{ACGP03} by considering the quiver with a single vertex $a$ and a single arrow $v$ decorated by the vector bundle $M_a = L\dual$:
\begin{equation}\label{eq:VW}
i \Lambda_\omega F_\fE^{1,1} + [\vphi, \vphi^{*_h}] = \tau \id_{\fE}
\end{equation}

where $\varphi^{*_h}$ is the adjoint of $\varphi$ (considered as an endomorphism $\fE \stackrel{[-,\vphi]}{\lra} \fE \otimes L$) with respect to the hermitian metric $h_\fE$ on $\fE = \ad(E_G)$ and the hermitian metric on $L$, and $\tau=\tau_\fE \in \RR$. Note that we can consider $\vphi^{*_h}$ as an element of $\Hom(\fE \otimes L, \fE)$.
Then $[\vphi, \vphi^{*_h}] = \tr_L(\vphi \circ \vphi^{*_h}) - \vphi^{*_h} \circ \vphi$ where $\tr_L$ is the trace map $\Hom(E \otimes L, E \otimes L) \to \Hom(E,E)$, so that $[\vphi, \vphi^{*_h}] \in \Hom(\fE, \fE)$.

\begin{rema}
When $L=K_X$, the above quiver vortex equation \eqref{eq:VW} is equivalent to the Yang-Mills-Higgs equation in \cite[Theorem 1.4]{Tan14}.
\end{rema}

\section{Infinitesimal Automorphisms of $L$-valued $G$-Higgs bundles} \label{sect:infaut}

In this section, we generalize the computation of infinitesimal automorphisms of $\Omega_X^{1}$-valued $G$-Higgs bundles in \cite[Theorem 3.1]{Bis09} to the case of $L$-valued $G$-Higgs bundles. 
For that, we establish an analogue of \cite[Lemma 3.4]{BS09}. 
For a $G$-Higgs bundle $\cE=(E_G,\vphi)$, we let $\cA(\cE)$ be the space of infinitesimal automorphism of $\cE$. Same as in \cite{Bis09}, we have a natural isomorphism $\cA(\cE) \cong \set{ \tau \in H^0(X, \ad(E_G)) \, | \, [\tau,\vphi]=0}$.

In this section, any section $\tau \in H^0(X,\ad(E_G))$ is sometimes considered as its induced endomorphism $\ad(E_G) \stackrel{[-,\tau]}{\lra} \ad(E_G)$, by an abuse of notation.

\begin{lemm} \label{lem:parallel}
Let $(E_G, \varphi)$ be a stable $L$-valued $G$-Higgs bundle on $X$. Let $\fE=\ad(E_G)$ and let $h_\fE$ be the unique admissible Hermitian Yang-Mills-Higgs metric satisfying \eqref{eq:VW}. Let $\cF := \cE nd(\fE)$ and let $s \in H^0(X, \cF)$ be a holomorphic 
endomorphism of $\fE$ such that $[s, \varphi] = 0$. (Here, $\vphi$ is considered as its induced endomorphism.) Then $s$ is parallel with respect to the Chern connection $D_\cF$.
\end{lemm}

\begin{proof}
Consider the Laplacian $\Delta_d = d d^* + d^* d$ where $(-)^*$ denotes the adjoint compatible with the Hermitian metric. Consider the holomorphic normal local coordinates $\set{z_k, \bar{z}_k}$. Then the K\"ahler form is locally expressed by $\frac{i}{2}\sum_{k} dz_k \wedge d\bar{z}_k $. Then locally the Laplacian is defined by $\sum_k \rou_{z_k} \rou_{\bar{z}_k}$. Since the section $s$ is holomorphic, we have $D_{\bar{z}_k}s =0$ (Here $D = D_\cF$). Let $\bracket{-,-}$ be the Hermitian inner product via the Hermitian metric $h_\fE$. Let us abbreviate $\rou_{z_k},\rou_{\bar{z}_k},D_{z_k}, D_{\bar{z}_k}$ by $\rou_k, \rou_{\bar{k}}, D_k, D_{\bar{k}}$. Then locally we have:
\begin{align*}
\Delta_d |s|^2 & = \sum_k \rou_k \rou_{\bar{k}} \bracket{s,s} = \sum_k \rou_k ( \bracket{D_{\bar{k}}s, s} + \bracket{s, D_k s} ) = \sum_k \rou_k \bracket{s,D_k s} \\
& = \sum_k \bracket{D_k s, D_k s} + \bracket{s, D_{\bar{k}} D_k s}  = \sum_{k} |D_k s|^2 + \sum_k \bracket{s, D_{\bar{k}} D_k s } \, .
\end{align*}
Since $Ds = \sum_k (D_k s) dz_k$ and $\bracket{dz_i, dz_j} = \delta_{ij}$ (since $\set{z_k, \bar{z}_k}_k$ is the holomorphic local coordinate), we have $\sum_k|D_k s|^2 = \frac{1}{2}|Ds|^2$ since $|dz_i|^2 = 2$ for all $i$. 

On the other hand, locally we can write $F_\cF^{1,1} = \sum_k F_{k \bar{k}} dz_k \wedge d\bar{z}_k$ where $F_{k, \bar{k}} = [D_k,D_{\bar{k}}]$. Since $s$ is holomorphic, we have $D_{\bar{k}} D_k s = - F_{k, \bar{k}}s$. Thus $\Delta_d |s|^2 = |Ds|^2 - \bracket{s, (\sum_k F_{k,\bar{k}} ) s}$.

Moreover, from $F_\cF^{1,1} = \sum_k F_{k \bar{k}} dz_k \wedge d\bar{z}_k$, we have $i\Lambda_\omega F_\cF^{1,1} = 2 \sum_k F_{k, \bar{k}}$. Hence we have
\begin{align*}
2 \Delta_d |s|^2 = |D_\cF s|^2 - \bracket{i\Lambda_\omega F_\cF^{1,1} s, s}.
\end{align*}
Then by Lemma \ref{lemm:endconn} below, we obtain the following formula for a holomorphic section $s$: 
\begin{equation} \label{eq:bochner}
2\Delta_d(|s|^2) = |D_\cF s|^2 - \langle [i\Lambda_{\omega} F^{1,1}_\fE, s], s \rangle.
\end{equation}
We must evaluate the curvature term. Using the Yang-Mills-Higgs equation \eqref{eq:VW}, the curvature term in \eqref{eq:bochner} becomes:
\[
-\langle [i\Lambda_{\omega} F^{1,1}_\fE, s], s \rangle = \langle [ [\varphi, \varphi^{*}], s], s \rangle.
\]
Using the Jacobi identity and the hypothesis $[s, \varphi] = 0$, we have $[[\varphi, \varphi^{*}], s] = [\varphi, [\varphi^{*}, s]]$. Therefore we have:
\begin{equation}\label{eq:adj_target}
\langle [ [\varphi, \varphi^{*}], s], s \rangle = \langle  [\varphi, [\vphi^*,s] ], s \rangle.
\end{equation}

By Lemma \ref{lemm:jacobi1} below, we have
\begin{align}\label{eq:norm}
\langle  [\varphi, [\vphi^*,s] ], s \rangle = \bracket{ [\vphi^*,s], [\vphi^*,s] } = | [\vphi^*,s] |^2.
\end{align}
Thus, the Bochner formula \eqref{eq:bochner} becomes:
\[
2 \Delta_d(|s|^2) = |D_\cF s|^2 + |[\varphi^{*}, s]|^2 \geq 0.
\]
Thus $|s|^2$ is a subharmonic function on the compact manifold $X$. By the maximum principle, $|s|^2$ is constant. Hence $\Delta_d(|s|^2) = 0$, which implies $|D_\cF s|^2 = 0$, so that $D_\cF s = 0$.
\end{proof}

\begin{lemm}\label{lemm:endconn}
For any $s \in H^0(X,\cF)$, we have
$F_\cF s = [F_\fE,s]$.
\end{lemm}
\begin{proof}
Let $v$ be a local section of $\fE$. Then
$$D_{\fE}(s(v))=D_{\cF}(s)(v)+s(D_{\fE}(v)).$$
Applying $D_{\fE}$ on both sides,
$$F_{\fE}(s(v))=D_{\fE}(D_{\cF}(s)(v))+D_{\fE}(s(D_{\fE}(v))).$$
By the (generalized) Leibniz rule, we see that
$$D_{\fE}(D_{\cF}(s)(v))=D_{\cF}(D_{\cF}(s))(v)-D_{\cF}(s)\wedge D_{\fE}(v)$$
and
$$D_{\fE}(s(D_{\fE}(v)))=D_{\cF}(s)\wedge D_{\fE}(v)+s(D_{\fE}(D_{\fE}(v))).$$
Thus
$$F_{\fE}(s(v))=D_{\cF}(D_{\cF}(s))(v)+s(D_{\fE}(D_{\fE}(v)))=F_{\cF}(s)(v)+sF_{\fE}(v),$$
that is,
$$F_{\cF}(s)(v)=[F_{\fE},s](v).$$
\end{proof}

\begin{lemm}\label{lemm:jacobi1}
Let $E$ be a vector bundle over $X$ with metric and let $F = \cE nd(E)$ be the endomorphism bundle equipped with the metric induced from $E$. For a point $x \in X$ and $A,B,C \in \End(E|_x)$, we have
\[
\bracket{[A,B],C}_x = \bracket{B,[A^*,C]}_x.
\] 
\end{lemm}
\begin{proof}
Note that $\bracket{\alpha,\beta}_x = \tr(\alpha \beta^*)$. Therefore
\begin{align*}
\bracket{[A,B],C}_x & = \bracket{AB-BA,C}_x = \tr(ABC^*) - \tr(BAC^*) = \tr(BC^* A) - \tr(BAC^*) \\
& = \tr(B (A^* C)^* ) - \tr(B(CA^*)^*) = \tr(B [A^*,C]^*) = \bracket{B, [A^*,C]}_x \, . 
\end{align*}
\end{proof}

Now, we are ready to prove the following theorem about infinitesimal automorphism of stable $L$-valued principal $G$-Higgs bundles.

\begin{theo} \label{thm:automorphisms}
Let $\cE:=(E_G,\vphi)$ be a \textbf{stable} $L$-valued principal $G$-Higgs bundle on a compact K\"ahler manifold $X$. Then the infinitesimal automorphism group is isomorphic to the center of the Lie algebra:
\[
\mathcal{A}(\mathcal{E}) = \mathfrak{z}(\mathfrak{g}).
\]
\end{theo}

\begin{proof}
Again let $\fE=\ad(E_G)$. Since $\mathcal{E}$ is stable, Theorem \ref{thm:stabcompare} in the next section (an analogue of \cite[Theorem 2.6]{AB01}) guarantees that $(\fE, \vphi)$ is polystable. Then, \cite[Theorem 3.1]{ACGP03} guarantees the existence of an admissible Hermitian metric $h$ on $\fE$ satisfying the Yang-Mills-Higgs equation \eqref{eq:VW}. Then this metric induces a Chern connection $D_\fE$ on the adjoint bundle $\fE$. 

By definition, $\mathcal{A}(\mathcal{E})$ consists of holomorphic sections $\xi \in H^0(X, \ad(E_G))$ satisfying $[\varphi, \xi] = 0$. Let $\xi \in \mathcal{A}(\mathcal{E})$. Let $\cE = \ad(E_G)$. Since $[\varphi, \xi] = 0$, the commutator relation for the induced endomorphisms $[\vphi,\xi]=0$ also holds. Thus Lemma \ref{lem:parallel} applies, and hence we have $D_\cF \xi = 0$ where $\xi$ is considered as its induced endomorphism. This directly implies $D_\fE \xi = 0$. In other words $\xi$ is parallel with respect to $D_\fE$.

Assume that $\xi \notin \fz(\fg)$. By Corollary \ref{coro:indchern}, there exists a connection $D_E$ on $E_G$ compatible with the holomorphic structure which induces the connection $D_\fE$ on the adjoint bundle $\fE$ of $E_G$. 
Then, same as the proof of \cite[Theorem 3.1]{Bis09}, we can find a parabolic subgroup $P \subset G$ and its holomorphic reduction, $E_P \subset E_G$, 
which is preserved by the connection $D_E$ on $E_G$.
Let us denote the adjoint bundle $\ad(E_P) := E_P \times_P \fp$ by $\fE_P$. Note that $\fE_P$ becomes a subbundle of $\fE$. 
Again, parallel to \cite[Theorem 3.1]{Bis09}, we can check that the Higgs field $\vphi \in H^0(X,\fE \otimes L)$ restricts to $\vphi_P := \vphi|_{\fE_P} \in H^0(X,\fE_P \otimes L)$. (See Remark \ref{rema:holred} to see how we can construct the holomorphic reduction.)

We now derive a contradiction by using the slope stability. Let us recall the Yang-Mills-Higgs equation for $D_\fE$:
\begin{align*} 
i \Lambda_\omega F_\fE^{1,1} + [\vphi, \vphi^{*_h}] = \tau \id_{\fE}.
\end{align*}
The degree of the bundle $\fE = \ad(E_G)$ is given by the integral of the trace of the (1,1)-component of its curvature:
\begin{align}\label{eq:curvdeg}
\deg(\fE) = \frac{1}{\Vol(X)}\int_X \tr(i\Lambda_\omega F^{1,1}_{\fE}) \frac{\omega^{n}}{n!}.
\end{align}
On the other hand, we observe that $\frac{1}{\Vol(X)}\int_X \tr\left( [\varphi, \varphi^{*_h}] \right) \frac{\omega^{n}}{n!} = 0$. Moreover, since $G$ is reductive, $\fE \cong \fE\dual$. Therefore, $\deg(\fE) = 0$, so that we obtain $\frac{1}{\mathrm{Vol}(X)}\int_X \rank(\fE) \frac{\tau \omega^{n}}{n!} = \rank(\fE) \tau = 0$ from \eqref{eq:VW}. Thus $\tau=0$.

Since the connection $D_E$ restricts to a connection $D_P$ on $E_P$, the connection $D_\fE$ also restricts to a connection $D_{\fE_P}$ on $\fE_P$. Thus, the equation \eqref{eq:VW} restricts over $\fE_P$:
\[
i \Lambda_\omega F_{D_{\fE_P}}^{1,1} + [\vphi, \vphi^{*_h}] = 0.
\]
Here we used $\tau = 0$. Therefore we have:
\[
\deg(\fE_P) = \frac{1}{\Vol(X)}\int_X \tr(i\Lambda_\omega F^{1,1}_{D_{\fE_P}}) \frac{\omega^{n}}{n!} = \frac{-1}{\Vol(X)} \int_X \tr( [\varphi, \varphi^{*_h}]) \frac{\omega^{n}}{n!} = 0
\]
where the last equality follows since $\tr([A,B])=0$. 
But since $\cE = (E_G,\vphi)$ is stable, we must have
\[
0 = \deg(\fE) > \deg(\fE_P).
\]
But in the above we proved that $\deg(\fE_P) = 0$, which leads to a contradiction. Therefore, the assumption that $\xi \notin \mathfrak{z}(\mathfrak{g})$ is false, so that we have $\mathcal{A}(\mathcal{E}) \subset \mathfrak{z}(\mathfrak{g})$. On the other hand, it is clear that $\fz(\fg) \subset \cA(\cE)$, so that we obtain $\cA(\cE) = \fz(\fg)$. 
\end{proof}

\begin{rema} \label{rema:holred}
We briefly review how to construct the holomorphic reduction $E_P \subset E_G$ in the same way as in \cite{Bis09}. Let $\Ad(E_G) := E_G \times _G G$ via the conjugate action of $G$ on $G$ itself. Let $x_0$ be a fixed point in $X$. Note that $\Ad(E_G)_{x_0}$ is isomorphic to the automorphism group of $(E_G)_{x_0}$ which commutes with the action of $G$ on $(E_G)_{x_0}$. Let $H$ be the subgroup of $\Ad(E_G)_{x_0}$ consisting of automorphisms that come from the holonomy of the connection $D_E$ on $E_G$.

Note that there is an isomorphism $\vphi: \Ad(E_G)_{x_0} \cong G$. Although these isomorphisms are not unique, the conjugacy class of $\vphi(P_0) \subset G$ is well defined. Let us fix a parabolic subgroup $P$ in this conjugacy class. 

Note that $\tau$ is parallel with respect to the connection $D_\fE$. Then $H$ fixes $\tau(x_0) \in \fE_{x_0}$. ($\Ad(E_G)_{x_0}$ acts on $\fE_{x_0}$.) Since $\tau \notin \fz(\fg)$, the centralizer $Z(\tau(x_0)) \subset \Ad(E_G)_{x_0}$ of $\tau(x_0)$ contains $H$ and is contained in a maximal parabolic subgroup $P_0$ by \cite[Proposition 2.1]{Ram75}. Let us define $Y_{x_0} \subset (E_G)_{x_0}$ by
\[
Y_{x_0} := \set{y \in (E_G)_{x_0}\, | \, p(y,P) = P_0}, \quad p: E_G \times G \to \Ad(E_G) \textrm{ is the projection.}
\]
Then we define $Y \subset E_G$ to be the submanifold consisting of $y \in E_G$ such that for a path $\gamma$ connecting $\pi(y)$ and $x_0$, where $\pi : E_G \to X$ is the projection, the parallel transport of $y$ along $\gamma$ via the connection $D_E$ lies in $Y_{x_0}$.

Note that since the holonomy group $H$ is contained in $P_0$, the above definition of $Y$ does not depend on the choice of the path $\gamma$, so that $Y$ is well-defined. Then we can show that $Y$ is in fact a holomorphic reduction of $E_G$ for the parabolic subgroup $P \subset G$.
\end{rema}

\section{Stability of $(E_G,\vphi)$ and $(\ad(E_G),\vphi)$} \label{sect:stability2}

In this section, we will show that if an $L$-valued $G$-Higgs bundle $(E_G,\vphi)$ is stable, then the adjoint $L$-valued Higgs bundle $(\ad(E_G),\vphi)$ (Here $\vphi$ stands for the morphism $[\vphi,-]: \ad(E_G) \to \ad(E_G)\otimes L$ by an abuse of notation) is polystable.

\begin{lemm}\label{tensor product of ss Higgs bundles is ss}
If $L$-valued Higgs bundles $(E_{1},\vphi_{E_{1}})$ and $(E_{2},\vphi_{E_{2}})$ are semistable, then $(E_{1}\otimes E_{2},\vphi_{E_{1}\otimes E_{2}})$ is also semistable, where $\vphi_{E_{1}\otimes E_{2}}=\id_{E_1}\otimes \vphi_{E_2} + \vphi_{E_1}\otimes \id_{E_2}$. 
\end{lemm}
\begin{proof}
If $(E_{1},\vphi_{E_{1}})$ and $(E_{2},\vphi_{E_{2}})$ are polystable, then $(E_{1}\otimes E_{2},\vphi_{E_{1}\otimes E_{2}})$ is also polystable by \cite[Proposition 5.3]{ACGP03b}. \footnote{We cite the arXiv version here since Proposition 5.3 only appear in the arXiv version.}

Assume that $(E_{1},\vphi_{E_{1}})$ is polystable and $(E_{2},\vphi_{E_{2}})$ is only semistable. Let
$$(0,0)=(F_{0},\vphi_{F_{0}})\subset(F_{1},\vphi_{F_{1}})\subset(F_{2},\vphi_{F_{2}})\subset\cdots\subset(F_{l-1},\vphi_{F_{l-1}})\subset(F_{l},\vphi_{F_{l}})=(E_{2},\vphi_{E_{2}})$$
be the filtration of $(E_{2},\vphi_{E_{2}})$, where $(F_{i}/F_{i-1},\vphi_{F_{i}/F_{i-1}})$ is the unique maximal polystable subpair of $(E_{2}/F_{i-1},\vphi_{E_{2}/F_{i-1}})$. Tensoring with $(E_{1},\vphi_{E_{1}})$, we have the filtration
\begin{multline}
(E_{1}\otimes F_{1},\vphi_{E_{1}\otimes F_{1}})\subset(E_{1}\otimes F_{2},\vphi_{E_{1}\otimes F_{2}})\subset\cdots\subset(E_{1}\otimes F_{l-1},\vphi_{E_{1}\otimes F_{l-1}}) \\ 
\subset(E_{1}\otimes F_{l},\vphi_{E_{1}\otimes F_{l}})=(E_{1}\otimes E_{2},\vphi_{E_{1}\otimes E_{2}})
\end{multline}
of $(E_{1}\otimes E_{2},\vphi_{E_{1}\otimes E_{2}})$. Since $(F_{j}/F_{j-1},\vphi_{F_{j}/F_{j-1}})$ is polystable with slope $\mu(E_{2})$ for all $j\in[1,l]$, $(E_{1}\otimes(F_{j}/F_{j-1}),\vphi_{E_{1}\otimes(F_{j}/F_{j-1})})$ is polystable with slope $\mu(E_{1}\otimes E_{2})$ for all $j\in[1,l]$ by the previous observation. Since any extension of a semistable Higgs bundle by a semistable Higgs bundle of the same slope is semistable, $(E_{1}\otimes E_{2},\vphi_{E_{1}\otimes E_{2}})$ is semistable.

Assume that both $(E_{1},\vphi_{E_{1}})$ and $(E_{2},\vphi_{E_{2}})$ are semistable. Since $(F_{j}/F_{j-1},\vphi_{F_{j}/F_{j-1}})$ is polystable, $(E_{1}\otimes(F_{j}/F_{j-1}),\vphi_{E_{1}\otimes(F_{j}/F_{j-1})})$ is semistable for all $j\in[1,l-1]$ by the previous observation. Since any extension of a semistable Higgs bundle by a semistable Higgs bundle of the same slope is semistable, we conclude that $(E_{1}\otimes E_{2},\vphi_{E_{1}\otimes E_{2}})$ is semistable.
\end{proof}

\begin{defi}
For a Higgs bundle $(E,\vphi)$, consider the Harder-Narasimhan filtration:
\[
(0,0) = (E^0, \vphi^0) \subset (E^1,\vphi^1) \subset \cdots \subset (E^\ell, \vphi^\ell) = (E,\vphi)
\]
where $E^i \subset E^{i+1}$ and $\vphi^{i+1}|_{E^i} = \vphi^i$ for all $i$. Then $\mu_{\max}((E,\vphi))$(resp. $\mu_{\min}((E,\vphi))$) is defined by
\[
\mu_{\max}(\text{resp. } \mu_{\min})((E,\vphi)) := \max (\text{resp. } \min)\set{ \mu(E^{i} / E^{i-1}) \, | \, i=1,2,\dots,\ell }.
\]
\end{defi}

\begin{prop}[{cf. \cite[Proposition 2.9]{AB01}}]\label{mumin and mumax}
For two Higgs bundles $(E_{1},\vphi_{E_{1}})$ and $(E_{2},\vphi_{E_{2}})$ on $X$,
$$\mu_{\min}((E_{1}\otimes E_{2},\vphi_{E_{1}\otimes E_{2}}))=\mu_{\min}((E_{1},\vphi_{E_{1}}))+\mu_{\min}((E_{2},\vphi_{E_{2}}))$$
and
$$\mu_{\max}((E_{1}\otimes E_{2},\vphi_{E_{1}\otimes E_{2}}))=\mu_{\max}((E_{1},\vphi_{E_{1}}))+\mu_{\max}((E_{2},\vphi_{E_{2}})).$$
\end{prop}
\begin{proof}
It is an immediate consequence from Lemma \ref{tensor product of ss Higgs bundles is ss} that the Harder-Narasimhan filtration of $(E_{1}\otimes E_{2},\vphi_{E_{1}\otimes E_{2}})$ is obtained by the tensor product of the subpairs in the Harder-Narasimhan filtration of $(E_{1},\vphi_{E_{1}})$ with the subpairs in the Harder-Narasimhan filtration of $(E_{2},\vphi_{E_{2}})$. 
Precisely, if
$$(0,0)=(E_{j}^{0},\vphi_{E_{j}^{0}})\subset(E_{j}^{1},\vphi_{E_{j}^{1}})\subset(E_{j}^{2},\vphi_{E_{j}^{2}})\subset\cdots\subset(E_{j}^{l_{j}-1},\vphi_{E_{j}^{l_{j}-1}})\subset(E_{j}^{l_{j}},\vphi_{E_{j}^{l_{j}}})=(E_{j},\vphi_{E_{j}})$$
is the Harder-Narasimhan filtration of $(E_{j},\vphi_{E_{j}})$ for $j=1,2$, then $(E_{1}^{1}\otimes E_{2}^{1},\vphi_{E_{1}^{1}\otimes E_{2}^{1}})$ is the maximal semistable subpair of $(E_{1}\otimes E_{2},\vphi_{E_{1}\otimes E_{2}})$. \\
Similarly, $\left( (E_{1}/E_{1}^{l_{1}-1})\otimes(E_{2}/E_{2}^{l_{2}-1}),\vphi_{(E_{1}/E_{1}^{l_{1}-1})\otimes(E_{2}/E_{2}^{l_{2}-1})} \right)$ is the maximal semistable quotient of $(E_{1}\otimes E_{2},\vphi_{E_{1}\otimes E_{2}})$. Thus we get the result.
\end{proof}

\begin{lemm}[{cf. \cite[Lemma 4.7]{AB01}}]\label{lemm:adHiggsss}
An $L$-valued principal $G$-Higgs bundle $(E_{G},\vphi)$ is semistable if and only if the associated adjoint Higgs bundle $(\ad(E_{G}), \vphi)$ is semistable.
\end{lemm}
\begin{proof}
By Lemma \ref{lemm:reduction} below, it is enough to prove under the assumption that $G$ is semisimple with trivial center.
That is, the adjoint representation of $G$ is faithful. 

Suppose that $(E_{G},\vphi)$ is not semistable. Then there exists a big open subset $U$, a maximal parabolic subgroup $P$ of $G$ and its holomorphic reduction $E_P \subset E_G|_{U}$ over $U$ such that the Higgs field $\vphi|_{U} \in H^0(U, \ad(E_G) \otimes L)$ restricts to $\vphi|_{U} \in H^0(U, \ad(E_P)\otimes L)$ and we have the inequality $\deg(\ad(E_G|_{U})) < \deg(\ad(E_P))$. 

Let $E_{0}=\ad(E_{P})$ which is a subbundle of $\ad(E_{G}|_{U})$. 
Since $G$ is semisimple, using the Killing form on $\fg$, we have $\ad(E_{G}|_{U}) \cong \ad(E_{G}|_{U})\dual$, that is, $\deg(\ad(E_{G}|_{U}))=0$. 
Since $\codim (X \setminus U) \ge 2$, $\ad(E_P)$ extends to an analytic subsheaf $\cF$ of $\ad(E_{G})$ and we can easily check that $\vphi(\cF) \subset \cF \otimes L$. Since $\deg(\ad(E_G)) = \deg(\ad(E_G|_{U}))$ and $\deg(\ad(E_P)) = \deg(\cF)$, we have $\deg(\ad(E_G)) < \deg(\cF)$. Thus $(\cF,\vphi)$ is a destabilizing subobject of $(\ad(E_{G}),\vphi)$, so that $(\ad(E_{G}),\vphi)$ is not semistable.

Conversely, suppose that $(\ad(E_{G}),\vphi)$ is not semistable. Let
\begin{equation}\label{HN}0=E_{0}\subset E_{1}\subset\cdots\subset E_{k-1}\subset E_{k}=\ad(E_{G})\end{equation}
be the Harder-Narasimhan filtration of $(\ad(E_{G}),\vphi)$, where $E_{j}$ is $\vphi$-invariant. Since each $E/E_{j}$ is torsion free, we know that they are locally free outside an analytic subset of codimension at least two. Let $V$ be the union of all these $k-1$ analytic subsets of codimension at least two. The complement $X\setminus V$ will be denoted by $U$.

For simplicity, denote $\ad(E_{G})$ by $E$. For any $x\in U$ consider $E_{j,x}^{\perp}=\{v\in E_{x}\,|\,\langle v,E_{j,x}\rangle=0\}$, where $\langle-,-\rangle$ is the Killing form on $E_{x}\cong\fg$. Let $E_{j}^{\perp}$ be the subsheaf defined by the kernel of the composition of the morphism:
\[
E \to E\dual \to E_j\dual
\]
where the morphism $E \to E\dual$ is the morphism induced by the Killing form. Then, the fiber of $E_j^\perp$ over any $x\in U$ is $E_{j,x}^{\perp}$. Note that $E_j^\perp$ is a saturated subsheaf. Since $G$ is semisimple, the Killing form is nondegenerate, and then we have $E_{j}^{\perp}\cong(E/E_{j})\dual$ over $U$.

Note that for the Killing form, we have $\bracket{[A,B],C} = \bracket{A,[B,C]}$. Moreover, by an abuse of notation, we let $\vphi : E \to E\otimes L$ be the endomorphism by taking the adjoint action of the Higgs field $\vphi \in H^0(X,E\otimes L)$, $[-,\vphi]$. Therefore, since $E_j$ is $\vphi$-invariant(i.e. $\vphi(E_j) \subset E_j \otimes L$), $E_j^\perp$ is also $\vphi$-invariant. 
Hence, by setting $W_{j}:=E_{k-j}^{\perp}$ 
we obtain the following filtration of $E$ over $U$
$$0=W_{0}\subset W_{1}\subset\cdots\subset W_{k-1}\subset W_{k}=E,$$
where $W_{j}$ is $\vphi$-invariant. We know that the dual of a semistable Higgs pair is again semistable.
Since $E\cong E\dual$, the filtration of $E$ over $U$ by $W_{j}$ coincides with the Harder-Narasimhan filtration (\ref{HN}). In other words, we have $E_{j}=E_{k-j}^{\perp}$ on $U$ for all $j\in[0,k]$.

Therefore, (\ref{HN}) is of the following form
$$0=E_{-l-1}\subset E_{-l}\subset E_{-l+1}\subset\cdots\subset E_{-1}\subset E_{0}\subset E_{1}\subset\cdots\subset E_{l-1}\subset E_{l}=E,$$
where $E_{-j}$ is orthogonal to $E_{j-1}$ for the Killing form and $E_{j}$ is $\vphi$-invariant.

Let $f:E_{0}\otimes E_{0}\to E/E_{0}$ be the composition of the Lie bracket operation with the natural projection $E\to E/E_{0}$. Note that $\vphi_{E_{0}\otimes E_{0}} = \vphi_{E_0}\otimes\id_{E_{0}}+\id_{E_{0}}\otimes\vphi_{E_0}$. Since
$$
[[\vphi,x],y]+[x,[\vphi,y]]=[\vphi,[x,y]]
$$
for any sections $x,y\in E_{0}$ by the Jacobi identity, we have
$$
f\in H^{0}(X,\cH om((E_{0}\otimes E_{0},\vphi_{E_{0}\otimes E_{0}}),(E/E_{0},\vphi_{E/E_{0}}))),
$$
where $\vphi_{E/E_{0}}$ is induced from $\vphi$. Using Proposition \ref{mumin and mumax}, we have
$$\mu_{\min}((E_{0}\otimes E_{0},\vphi_{E_{0}\otimes E_{0}}))=2\mu_{\min}((E_{0},\vphi_{E_{0}}))=2\mu(E_{0}/E_{-1})$$
and
$$
\mu_{\max}((E/E_{0},\vphi_{E/E_{0}}))=\mu(E_{1}/E_{0}).
$$
Since $E_{-1}$ is the orthogonal part of $E_{0}$, the Killing form induces a nondegenerate quadratic form on $E_{0}/E_{-1}$. Consequently, we have $E_{0}/E_{-1}\cong(E_{0}/E_{-1})\dual$, which implies that $\mu(E_{0}/E_{-1})=0$. (Note that for any torsion free sheaf $F$ we have $\mu(F)=-\mu(F^{\vee})$.)

Since we have
$$
\mu_{\min}((E_{0}\otimes E_{0},\vphi_{E_{0}\otimes E_{0}}))=2\mu(E_{0}/E_{-1})=0>\mu(E_{1}/E_{0})=\mu_{\max}((E/E_{0},\vphi_{E/E_{0}})),
$$
it follows from the analogous statement of \cite[Proposition 2.8(1)]{AB01} that
$$H^{0}(X,\cH om((E_{0}\otimes E_{0},\vphi_{E_{0}\otimes E_{0}}),(E/E_{0},\vphi_{E/E_{0}})))=0.$$
In particular, $f=0$; that is, $E_{0}$ is closed under the Lie algebra structure of the fibers of $E$ compatible with $\vphi$.

Consider the following homomorphism
$$f_{j}:E_{-j}\otimes E_{-1}\to E/E_{-j-1},$$
where $j\ge0$, defined using the Lie bracket operation and the projection of $E$ to $E/E_{-j-1}$. Repeating the above argument and using the property of the Harder-Narasimhan filtration that $\mu(E_{i}/E_{i-1})>\mu(E_{i+1}/E_{i})$ we deduce that $f_{j}=0$ in
$$
H^{0}(X,\cH om((E_{-j}\otimes E_{-1},\vphi_{E_{-j}\otimes E_{-1}}),(E/E_{-j-1},\vphi_{E/E_{-j-1}}))).
$$
In other words, we have
\begin{align}\label{eq:inclusion}
[E_{-j},E_{-1}]\subset E_{-j-1}
\end{align}
for any $j\ge0$. Moreover, this is compatible with $\vphi$.

Using the above inclusion we conclude that $E_{-1}$ is a nilpotent Lie ideal of $E_{0}$ compatible with $\vphi$. We can complete the proof parallel to the proof of \cite[Proposition 2.10]{AB01} as follows.

By Lemma \ref{lemm:parabolic reduction}, we can construct a parabolic reduction $E_P$ of $E_G$ from $E_0$, where $P \subset G$ is a parabolic subgroup. Then the line bundle $\det(E_0)$ corresponds to the character $\chi_0: P \to \CC^*$, and as in the proof of \cite[Proposition 2.10]{AB01}, we can prove that $\chi_0$ is a dominant character. We can check that $E_0 \cong E_P \times_P \fp$ and since $E_0$ is $\vphi$-invariant, $(E_P, \vphi)$ is an $L$-valued $G$-Higgs subbundle of $E_G|_U$. 

Note that we can easily generalize \cite[Proposition 1.4]{AB01} for $L$-valued principal $G$-Higgs bundles. Then, since $\deg(E_P(\chi_0)) = \deg(\det(E_0)) = \deg(E_0) > 0$, the generalized version of \cite[Proposition 1.4]{AB01} implies that $E_G$ is not semistable, which leads to a contradiction.
\end{proof}

The following lemma is to justify the assumption made in the beginning of the proof of Lemma \ref{lemm:adHiggsss}.

\begin{lemm}\label{lemm:reduction}
Let $G$, $G'$ be reductive algebraic groups, and let $f:G\to G'$ be a surjective homomorphism with $\ker f\subset Z(G)$, where $Z(G)$ is the center of $G$. Let $(E,\vphi)$ be an $L$-valued principal $G$-Higgs bundle and $(E',\vphi')$ the $L$-valued $G'$-Higgs bundle obtained by the extension of the structure group by $f$ as follows:
$$E'=E\times_{G,f}G'\text{ and }\vphi'=df\circ\vphi$$
Then $(E,\vphi)$ is stable (resp. semistable) if and only if $(E',\vphi')$ is stable (resp. semistable).
\end{lemm}
\begin{proof}
Since $f:G\to G'$ is a surjective homomorphism with $\ker f\subset Z(G)$, it is easy to see that the induced morphism $\bar{f}:E/P\to E'/P'$ is an isomorphism, where $P,P'$ are maximal parabolic subgroups of $G,G'$ respectively such that $f^{-1}(P')=P$. Then we have the induced isomorphism $\tilde{f}:E\times_{P,\Ad_{G}}(\fg/\fp)\to E'\times_{P',\Ad_{G'}}(\fg'/\fp')$, that is, $\tilde{f}:T_{E/P}\to T_{E'/P'}$. Further, since $df(\fp)=\fp'$, $\vphi\in H^{0}(X,\ad(E_{P})\otimes L)$ implies $\vphi'\in H^{0}(X,\ad(E'_{P'})\otimes L)$.

Assume that $(E,\vphi)$ is not stable. Then there is a maximal parabolic subgroup $P\subset G$ and a reduction of the structure group $\sigma:U\to E/P$ with $\vphi\in H^{0}(X,\ad(E_{P})\otimes L)$ such that $\deg\,\sigma^{*}(T_{E/P})\le0$. Let $\sigma'=\bar{f}\circ\sigma$ so that $\sigma':U\to E'/P'$ is a reduction of the structure group. Then $\vphi'\in H^{0}(X,\ad(E'_{P'})\otimes L)$, $\sigma'^{*}(T_{E'/P'})=\sigma^{*}(\bar{f}^{*}(T_{E'/P'}))=\sigma^{*}(T_{E/P})$ and thus $\deg\,\sigma'^{*}(T_{E'/P'})\le0$, which implies that $(E',\vphi')$ is not stable.

Assume that $(E,\vphi)$ is stable. Let $\sigma':U\to E'/P'$ be a reduction of the structure group with $\vphi'\in H^{0}(X,\ad(E'_{P'})\otimes L)$. It is enough to show that $\sigma'=\bar{f}\circ\sigma$ for some reduction of the structure group $\sigma:U\to E/P$. But this is verified from the cohomological argument of the proof of \cite[Proposition 7.1]{Ram75}. The semistability is also preserved by the same argument.
\end{proof}

The following lemma gives the reduction of the structure group that we are looking for.

\begin{lemm}[{cf. \cite[Lemma 2.11]{AB01}}]\label{lemm:parabolic reduction}
Over the open set $U$ of $X$, 
$E_{0}$ is a bundle of parabolic subalgebras such that $\vphi\in H^{0}(X,E_{0}\otimes L)$, and it gives a reduction $\sigma :U\to E_{G}/P$ of the structure group of $E_{G}$ to a parabolic subgroup $P$ of $G$.
\end{lemm}
\begin{proof}
By \eqref{eq:inclusion}, $E_{-1}$ is a bundle of nilpotent subalgebras. By \cite[Equation (4) on page 216]{AB01}, $E_{-1}$ is the radical of $E_{0}$. Since $E_{0}$ is the normalizer $N_{E_{0}}(E_{-1})$ of $E_{-1}$, $E_{0}$ is a bundle of parabolic subalgebras by \cite[Lemma II-2-(ii)]{Falt93}. Since
$$[[\vphi,x],y]=-[y,[\vphi,x]]=-[x,[\vphi,y]]+[\vphi,[x,y]]\in E_{-1}$$
for any section $x\in E_{0}$ and any section $y\in E_{-1}$ by the Jacobi identity, $\vphi_{E_{0}}$ is induced from $\vphi_{E_{-1}}$ and then $\vphi\in H^{0}(X,E_{0}\otimes L)$. The proof of the remaining part is the same as that of \cite[Lemma 2.11]{AB01}.
\end{proof}

Since $(E,\vphi)=(\ad(E_{G}),\vphi)$ is semistable, we have its socle, i.e., maximal polystable subpair. Let $(S,\vphi_{S})$ be the socle of $(E,\vphi)$. We can reduce the structure group of $(E_{G},\vphi)$ using the socle $(S,\vphi_{S})$ in the following lemma.

\begin{lemm}[{cf. \cite[Proposition 2.12]{AB01}}]\label{existence of parabolic subalgerba bundle}
There is a subalgebra bundle $\fP$ of $\ad(E_{G})$ over an open subset $U\subseteq X$, with $\codim(X\setminus U)\ge 2$, such that $\vphi\in H^{0}(X,\fP\otimes L)$, $\deg\fP=0$ and the fibers of $\fP$ over $U$ are isomorphic to the Lie algebra $\fp$ of a parabolic subgroup $P$ of $G$. 
\end{lemm}
\begin{proof}
Since the normalizer of a $\vphi$-invariant subalgebra bundle is also $\vphi$-invariant by the Jacobi identity, we use the same argument as the proof of  \cite[Proposition 2.12]{AB01}.
\end{proof}

\begin{theo}[{cf. \cite[Theorem 4.8]{AB01}}] \label{thm:stabcompare}
If $(E_{G},\vphi)$ is stable, then $(\ad(E_{G}),\vphi)$ is polystable.
\end{theo}
\begin{proof}
Assume that $(E_{G},\vphi)$ is stable. By Lemma \ref{lemm:adHiggsss}, $(\ad(E_{G}),\vphi)$ is semistable. Suppose that $(\ad(E_{G}),\vphi)$ is not polystable. Then the socle $(S,\vphi_{S})$ is properly contained in $(\ad(E_{G}),\vphi)$. By Lemma \ref{existence of parabolic subalgerba bundle}, there exists an open subset $U\subseteq X$, with $\codim(X\setminus U)\ge 2$ and a parabolic subalgebra bundle $\fP$ of $\ad(E_{G})$ over $U$ with $\vphi\in H^{0}(X,\fP\otimes L)$ and $\deg\fP=0$.

Then by Lemma \ref{lemm:parabolic reduction}, this parabolic subalgebra bundle defines a parabolic reduction $\sigma :U\to E_{G}/P$ of the structure group of $E_{G}$ to the parabolic subgroup $P$ and we have $\fP\cong \sigma^{*}(E_G \times_P \fp)$. Let $\chi:=\det(\ad_{P})$ be the determinant of the adjoint action of $P$ on its Lie algebra $\fp$. This character $\chi$ is dominant. On the other hand we have
$$\deg(E_{P}(\chi))=\deg(E_{P}(\fp))=\deg(\fP)=0.$$
This contradicts the stability of $(E_{G},\vphi)$. Consequently, $(\ad(E_{G}),\vphi)$ coincides with its socle $(S,\vphi_{S}):=\displaystyle \bigoplus_{i=1}^{k}(\cF_{i},\vphi_{\cF_{i}})$ over $U$, i.e.,
$$(\ad(E_{G}),\vphi)|_{U}\cong\bigoplus_{i=1}^{k}(\cF_{i},\vphi_{\cF_{i}})|_{U},$$
where $\cF_{i}$ are locally free and $(\cF_{i},\vphi_{\cF_{i}})|_{U}$ is a stable subpair of $(\ad(E_{G}),\vphi)|_{U}$. Since $\ad(E_{G})$ and $S$ are reflexive, $\cH om(\ad(E_{G}),S)$, $\cH om(\ad(E_{G}),S\otimes L)$, $\cH om(\ad(E_{G}),\ad(E_{G}))$, $\cH om(S,\ad(E_{G}))$, $\cH om(S,\ad(E_{G})\otimes L)$ and $\cH om(S,S)$ are all normal by \cite[Proposition 5.5.21, 5.5.23]{Ko87}, where a coherent sheaf $\cG$ on $X$ is said to be normal if for every open set $V$ in $X$ and every analytic subset $A$ of $V$ of codimension at least $2$, the restriction $\Gamma(V,\cG)\to\Gamma(V \setminus A,\cG)$ of sections is an isomorphism. So the following commutative diagram over $U$
$$\xymatrix{\ad(E_{G})|_{U}\ar[r]^{\cong}\ar[d]_{\vphi}&S|_{U}\ar[d]^{\vphi_{S}}\\ \ad(E_{G})\otimes L|_{U}\ar[r]^{\cong}&S\otimes L|_{U}}$$
extends over $X$. Thus the above holomorphic decomposition of $(\ad(E_{G}),\vphi)|_{U}$ extends to $X$, that is, $(\ad(E_{G}),\vphi)$ is polystable.
\end{proof}

\section{Application: Moduli stack of $L$-valued Higgs G-sheaves and the perfect obstruction theory.}\label{sect:POT}

In this section, we will assume that $X$ is a smooth projective variety. It is well known that the moduli stack of principal $G$-bundles on $X$, $\Bun_G(X)$ is an algebraic stack. Let $U_G$ be the universal $G$-bundle over $\Bun_G(X) \times X$. Then we can also consider the universal adjoint bundle $\ad(U_G) := U_G \times_G \fg$. Next, we consider the moduli stack of $L$-valued $G$-Higgs fields. We can consider the stack associated to the coherent sheaf $(p_1)_* \left( \ad(U_G) \otimes p_2^*L \right)$ where $p_1, p_2$ are the projections from $\Bun_G(X) \times X$ to $\Bun_G(X)$ and $X$, respectively. (See \cite[Section 3]{CL12} and \cite[Section 3]{KMMP25} for the definition of the stack associated to the coherent sheaf.) We denote this stack by $\Higgs_L(X)$.

Then, this stack is an algebraic stack, locally of finite presentation over $\Bun_G(X)$, with quasi-affine diagonal over $\Bun_G(X)$ by \cite[Theorem 4]{HR14}. Now, we consider the stable locus of the moduli space of $L$-valued $G$-Higgs fields, $\Higgs_L(X)^s$. Then we claim that it is a DM stack. 

\begin{prop}
When $G$ is semisimple, the stability condition in Definition \ref{def:stability of VW G-bundle} is an open condition. That is, $\Higgs_L(X)^s$ is an open substack of the stack of $L$-valued $G$-Higgs fields, $\Higgs_L(X)$.
\end{prop}
\begin{proof}
We will first show that the stability condition for $G$-bundles is an open condition. When $\vphi=0$, the stability condition in Definition \ref{def:stability of VW G-bundle} is reduced to the stability of principal $G$-bundles.

Let $S$ be a scheme of finite type over $\CC$ and let $E_G \to X \times S$ be a flat family of principal $G$-bundles. We will show that the locus $S^{\text{st}} = \{s \in S \mid E_G|_s \text{ is stable}\}$ is an open subset of $S$. $E_G|_s$ is not stable if and only if there exists a maximal parabolic subgroup $P \subset G$ and a reduction $\sigma: U \to E_s/P$ on an open set $U$ with $\text{codim}(X \setminus U) \ge 2$, such that:
$$ 
\deg \, \sigma^*(T_{E_s /P}) \le 0.
$$

On the other hand, the reduction $\sigma$ to the parabolic subgroup $P$ over $U$ is equivalent to a vector subbundle $\mathfrak{F}_\sigma \subset \mathfrak{E}_s|_U$ whose fibers are conjugate to the parabolic Lie algebra $\mathfrak{p} \subset \mathfrak{g}$, where $\mathfrak{E}_s = (E_G|_s ) \times_G \fg$ by Lemma \ref{lemm:genericpara}.

Note that we have $\sigma^*(T_{E_s/P}) \cong (\fE_s|_U) / \fF_\sigma$. Since the group $G$ is semisimple, the Killing form on the Lie algebra $\fg$ gives an isomorphism $\fE_s \cong \fE_s^{\vee}$. Thus $\deg \fE_s = 0$. Therefore, $\deg \, \sigma^*(T_{E_s /P}) \le 0$ if and only if $\deg \fF_\sigma \ge 0$.

Because $X$ is a smooth (hence normal) projective variety and $\text{codim}(X \setminus U) \ge 2$, the vector subbundle $\fF_\sigma$ extends uniquely to a saturated, torsion-free coherent subsheaf $\cF \subset \fE_s$ on $X$. Then we have $\deg(\cF) = \deg(\fF_\sigma) \ge 0$. Note that $\cF$ is saturated since $(\fE_s/\cF)|_U \cong (\fE_s|_U)/\fF_\sigma$ is locally free.

Therefore, $E_G|_s$ is not stable if and only if its adjoint bundle $\fE_s$ contains a saturated coherent subsheaf $\mathcal{F}$ such that $\deg(\cF) \ge 0$ and over a big open set $U$ (codimension of $X\setminus U \ge 2$) each fiber of $\cF|_U \subset \fE_s|_U 
$ is a parabolic subalgebra conjugate to $\mathfrak{p} \subset \fg$.

By \cite[Theorem 1.1, Proposition 1.8]{Sim94},
for the family of vector bundles $\fE$ over $X \times S$, flat over $S$, the set of all saturated subsheaves of $\fE$, $s \in S$ whose degrees are bounded from below ($\deg \ge 0$) forms a bounded family, denoted by $Y^{\text{dst}}$.

Consequently, there are only finitely many possible Hilbert polynomials $\Phi_1, \dots, \Phi_k$ that the quotient sheaf $\mathcal{Q} = \fE_s / \mathcal{F}$ can have, where $\cF$ is an element of the above family $Y^{\text{dst}}$. Because the set of possible Hilbert polynomials is finite, we can parameterize the entire search space algebraically. We construct the relative Quot scheme:$$ \mathcal{Q}uot := \coprod_{i=1}^k \text{Quot}_{X \times S / S}(\fE, \Phi_i) $$
which parameterizes all coherent quotients $q: \fE_s \twoheadrightarrow \mathcal{Q}$ possessing one of the destabilizing Hilbert polynomials $\Phi_1,\dots,\Phi_k$. Because $X$ is a projective variety,
the structure morphism $$ \pi: \mathcal{Q}uot \to S $$is a projective (and therefore proper) morphism by \cite[Theorem 2.2.4]{HT10}.

On the other hand, $\cQ uot$ parametrizes all subsheaves $\mathcal{F} = \ker(q)$. We want to consider the sublocus of $\cQ uot$ where $\mathcal{F}$ is isomorphic to the adjoint bundle of the parabolic reduction of $\cE_s$ over a big open set $U$. (That is, there exists a reduction $E_P$ of $E_G|_U$ for a parabolic subgroup $P \subset G$ such that $\cF|_U \cong E_P \times_P \fp$.)

We consider the following parabolic type condition on $\mathcal{Q} uot$. Let $[0 \to \cF \to \fE_s \to \cQ \to 0] \in \cQ uot$ be an element. The parabolic type condition is that over a big open set $U \subset X$, $\cF|_x \subset \fE_s|_x \cong \fg$ is a parabolic subalgebra conjugate to $\mathfrak{p}$. 

Inside the Grassmannian bundle $\operatorname{Gr}(r, \fE)$, the locus of subspaces that are parabolic subalgebras conjugate to $\mathfrak{p}$ forms a closed subvariety whose fiber over each $x \in U$ is isomorphic to $G/P$. It can be shown as follows. For each closed point $s \in S$ and a closed point $x \in U$ we can consider an orbit map $G \to \Gr(r,\fE_s|_x)$ defined by:
\[
\vphi: g \mapsto g \fp g^{-1}.
\]
Then we can easily see that the stabilizer of this orbit map is isomorphic to $P \subset G$. Therefore, there is an injective morphism $\bar{\vphi}: G/P \to \Gr(r,\fE_s|_x)$. Since $G/P$ is a generalized flag variety, it is projective. Therefore, $\bar{\vphi}$ is a closed immersion, so that the locus of subspaces that are parabolic subalgebras conjugate to $\fp$ is a closed subvariety and we denote it by $\Gr(\fp ,\fE)$. Note that it is in fact isomorphic to $E_G/P$. 

We claim that the condition that the generic fiber of $\cal{F}$ is contained in this $\Gr(\fp,\fE_s)$ gives us a closed condition by Lemma \ref{lemm:paraclosed}. Therefore, for each $\Phi_i$, there is a closed subscheme $\cQ_{\fp, \Phi_i} \subset \mathcal{Q} uot$ parametrizing the subsheaves $\cF \subset \fE_s$ such that over a certain big open set $U \subset X$, $\cF|_U$ is a vector subbundle of $\fE_s|_U$ and $\fF_{\sigma}|_x \subset (\cE_s)_x$ is isomorphic to a parabolic subalgebra conjugate to $\fp$ for all $x \in U$. Since $\cQ uot$ is a projective scheme, so is $\cQ_{\fp,\Phi_i}$. Therefore, the projection $\pi : \cQ_{\fp,\Phi_i} \to S$ is also projective (thus it is proper). Hence $\pi(\cQ_{\fp,\Phi_i})$ is a closed subscheme of $S$.

Note that there are only finitely many conjugacy classes of parabolic subalgebras when we fix the maximal torus of $T \subset G$ and the corresponding Cartan subalgebra $\fh \subset \fg$, which corresponds to a finite subset of the set of simple roots. Let $\fp_1,\dots,\fp_\ell$ be the collection of all parabolic subalgebras. 

Therefore, $S^{\text{dst}}:= \displaystyle \bigcup_{\fp_j,\Phi_i} \pi(\cQ_{\fp_j,\Phi_i}) \subset S$ is the locus where $\fE_s$ admits a destabilizing subsheaf $\cF$ whose restriction on a big open set $U \subset X$ is isomorphic to the adjoint of a principal $P$-bundle $E_P \subset E_G|_U$ for a parabolic subgroup $P \subset G$ by Lemma \ref{lemm:genericpara}. Therefore, by the above argument, the sublocus $S^{\text{st}}$ of $S$ where $E_G|_s$ is stable is equal to an open subset $S \, \setminus \, S^{\text{dst}}$. 

Next we will generalize the above proof to the general case, where the Higgs field is not necessarily zero. Over $X \times S$, we consider a flat family of Higgs $G$-sheaves $(E_G,\vphi)$, $\vphi \in H^0(X \times S, \fE \otimes L)$ where $\fE=E_G \times_G \fg$. We have the induced family of adjoint vector bundles with $L$-valued Higgs fields, $(\fE,\vphi)$, $\vphi : \fE \to \fE \otimes L$ defined by taking Lie bracket on the right. 

Then, in Lemma \ref{lemm:paraclosed}, we will modify the subscheme $Z_\fp$ by the sublocus satisfying the additional condition, that is, $\vphi(\cF) \subset \cF \otimes L \subset \fE_s \otimes L$. This condition is interpreted as follows. From the universal sequence, we can consider the universal morphism, defined as the following composition:
\[
v: \wtil{\cF} \stackrel{\wtil{\vphi}}{\lra} q_{X \times S}^*(\fE \otimes L) \to q_{X \times S}^*(\fE \otimes L)/(\wtil{\cF}\otimes L)  
\]
where $q_{X \times S} : \Quot_{X \times S/S}(\fE,\Phi_i) \times_S (X \times S) \to (X \times S)$ is the projection and $\wtil{\vphi}$ is the morphism by taking Lie bracket for the Higgs field on the right. Then we can consider a new subscheme $Z'_\fp \subset Z_\fp$ which is the zero locus of $v$, which is again closed. Then the above proof for the case when the Higgs field is zero, directly extends to the general case. 
\end{proof}

\begin{lemm}\label{lemm:genericpara} Let $\fE:= E_G \times_G \fg$ be an adjoint of a principal $G$-bundle over $X$ and let $\cF \subset \fE$ be a subsheaf whose fiber over any point $x$ in a big open set $U \subset X$, $\cF|_x$ is a parabolic subalgebra of $\fE|_x \cong \fg$ which is conjugate to $\fp$. Then there exists a $P$-bundle $E_P \subset E_G|_U$ such that $\cF|_U \cong E_P \times_P \fp$.
\end{lemm}
\begin{proof}
By the assumption, $\cF|_U$ gives a section $s : U \to \Gr(r,\fE)$ where $r = \dim_{\CC} \fp$. Moreover, the image of $s$ lies in $\Gr(\fp,\fE|_U) \cong (E_G/P)|_U$. This is a reduction of $E_G|_U$, where $E_P$ is defined by the pull-back of the principal $P$-bundle $E_G \to E_G/P$. We can easily check that $\cF|_U \cong E_P \times_P \fp$. 
\end{proof}

\begin{lemm}\label{lemm:paraclosed}
Inside the relative quot scheme $\cQ=\Quot_{X\times S/S}(\fE,\Phi_i)$, consider the sublocus of quotients $q: \fE_s \to \cQ$ where the kernel $\cF := \ker \, q$ satisfies the \textbf{parabolic condition} that over a big open set $U \subset X$, $\cF|_U$ is locally free and each fiber $\cF_x \subset \fE_s|_x \cong \fg$ is a Lie subalgebra which is a conjugate to a fixed parabolic subalgebra $\fp$ for each $x \in U$. Let us denote this sublocus by $Z_{\fp}$. Then $Z_{\fp}$ is a closed subscheme of $\Quot_{X\times S/S}(\fE,\Phi_i)$.
\end{lemm}
\begin{proof}
From the inclusion $\cF = \ker \, q \subset \fE_s$, there is a morphism $\wedge^r \cF \subset \wedge^r \fE_s$, where $r = \dim_\CC \fp$. As we have seen above, $\Gr(\fp,\fE_s|_x)$ is a closed subscheme of $\Gr(r,\fE_s|_x)$ and therefore, it is a closed subscheme in the Pl\"ucker embedding $\PP(\wedge^r \fE_s|_x)$. Then, $\Gr(\fp, \fE_s|_x)$ is defined by an ideal $I$ of the projective coordinate ring of $\PP(\wedge^r \fE_s|_x)$, and we may assume that $I$ is generated by homogeneous elements of degrees up to $d$, namely $I_1,\dots, I_d$. In other words, there are sections $s_i \in \Hom( \Sym^i (V), (I_i)\dual)$ where $V:=\wedge^r (\fE_s|_x)$ such that $\Gr(\fp, \fE_s|_x)$ is the zero locus of $s_1,\dots,s_d$.

From this fiberwise observation, we conclude that $\Gr(\fp, \fE) \subset \Gr(r, \fE)$ is the zero locus of the sections $s_i \in \Hom(\Sym^i(\wedge^r \fE), \cI_i\dual)\cong H^0\left( \Gr(r,\fE), \Hom_{\Gr(r,\fE)}(\Sym^i(\wedge^r \fE), p^*\cI_i\dual) \right)$. 

Over $\Quot_{X \times S/S}(\fE,\Phi_i) \times_S (X \times S)$, there is a universal sequence 
$$0 \to \wtil{\cF} \to q_{X\times S}^* \fE \to \wtil{Q} \to 0$$
where $q_{X\times S} : \Quot_{X \times S/S}(\fE,\Phi_i) \times_S (X \times S) \to (X \times S)$ is the projection. Consider the composition of the morphism
\[
\Sym^i(\wedge^r \wtil{\cF}) \to \Sym^i(\wedge^r q_{X\times S}^*\fE ) \to q_{X\times S}^* \cI_i\dual
\]
and denote it by $u_i$. The morphism $u_i$ is an element of $H^0(\cQ, \Hom_\cQ(\Sym^i(\wedge^r \wtil{\cF}), q_{X\times S}^* \cI_i\dual))$ 
where $\cQ:= \Quot_{X \times S/S}(\fE,\Phi_i)$. Let us show that $Z_\fp$ is a common zero locus of $u_1, \dots, u_d$. For $s \in S$, if $\cF \subset \fE_s$ satisfies the parabolic condition, then by the construction, $u_i|_{[\fE_s \surra \fE_s/\cF]} \in H^0(X,\cH om( \Sym^i (\wedge^r \cF), \cI_i\dual ) )$ vanishes over a big open set $U \subset X$. Since $\cI_i\dual$ is locally free, $\cH om(\Sym^i (\wedge^r \cF), \cI_i\dual)$ is torsion free. Hence $u_i|_{[\fE_s \surra \fE_s/\cF]} \in H^0(X,\cH om(\Sym^i (\wedge^r \cF), \cI_i\dual) ) = 0$ over the entire $X$. 

Conversely, assume that an element $[\fE_s \surra \fE_s/\cF]\in \cQ$ satisfies the condition that $u_i|_{[\fE_s \surra \fE_s/\cF]} \equiv 0$ for all $i=1,2,\dots,d$. Since $\cF$ is a subsheaf of the locally free sheaf $\fE_s$, it is torsion free. Therefore, there exists a big open set $U \subset X$ where $\cF|_U$ is locally free. Then for any $x \in U$, $\cF_x \subset \fE_s|_x \cong \fg$ is a parabolic subalgebra conjugate to $\fp$. Therefore, $Z_\fp \subset \cQ$ is indeed the common zero set of the sections $u_1, \dots, u_d$ and hence it is a closed subscheme of $\cQ = \Quot_{X \times S/S}(\fE,\Phi_i)$.
\end{proof}

We are now ready to state and prove the main result of this section.
\begin{prop}
For a semisimple Lie group $G$, $\Higgs_L(X)^s$ is a DM stack.
\end{prop}
\begin{proof}
By \cite[Theorem 4.6.4]{Alp26}, it is enough to show that every point of $\Higgs_L(X)^s$ has a discrete and reduced stabilizer group. It is enough to show that the sheaf of infinitesimal automorphisms, $H^{-1}(T_{\Higgs_L(X)^s}|_x) = 0$ for every point $x$ of $\Higgs_L(X)^s$ where $T_{\Higgs_L(X)^s}$ is the tangent complex. By the upper semi-continuity, it is enough to show that $H^{-1}(T_{\Higgs_L(X)^s}|_x) = 0$ for every closed point $x$. Let $x = \cE = [E_G, \vphi]$ be an arbitrary closed point of $\Higgs_L(X)^s$. But by Theorem \ref{thm:automorphisms}, the group of infinitesimal automorphisms $\cA(\cE)$ is equal to $\fz(\fg) = 0$ since $G$ is semisimple. 
\end{proof}

Furthermore, when $G$ is semisimple, $X$ is a smooth projective surface, and $L = K_X$, we will show that $\Higgs_L(X)^s$ is equipped with a perfect obstruction theory. By \cite[Theorem 4.1]{HT10}, \ it is known that there is a natural morphism (called deformation-obstruction theory) $\phi_B:   T_{\Bun_G(X)} \to E_B := R (p_1)_* \ad(U_G)[1]$ such that $H^0(\phi_B)$ is an isomorphism and $H^1(\phi_B)$ is injective.  Moreover, by \cite[Proposition 2.5, Proposition 3.1]{CL12}, it is also known that there is a relative perfect obstruction theory $\phi_{H/B}: T_{\Higgs_{K_X}(X) / \Bun_G(X)} \to E_{H/B}:=R (p_H)_* \pi^*(\ad(U_G) \otimes (p_2)^* K_X )$ where $p_H : \Higgs_{K_X}(X) \times X \to \Higgs_{K_X}(X)$ is the projection and $\pi : \Higgs_{K_X}(X) \times X  \to \Bun_G(X) \times X$ is the forgetful morphism.

On the other hand, there is a universal complex 
$$U_{\Higgs} := \left[ \pi^* \ad(U_G) \stackrel{U}{\lra} \pi^* (\ad(U_G) \otimes p_2^* K_X) \right]$$ 
over $\Higgs_{K_X}(X) \times X$, where the morphism $U$ is obtained by taking the Lie bracket on the right of the universal section of $\pi^* ((p_1)_* (\ad(U_G) \otimes p_2^* K_X) )$. Then there is a distinguished triangle 
\[
\pi^* \ad(U_G) \stackrel{U}{\lra} \pi^* (\ad(U_G) \otimes (p_2)^* K_X) \to U_\Higgs[1] \stackrel{+1}{\lra}
\]
Therefore, by taking $R(p_H)_*$, we have the distinguished triangle:
\[
q^*E_B[-1] \to E_{H/B} \to E_H \stackrel{+1}{\lra}
\]
where $q: \Higgs_{K_X}(X) \to \Bun_G(X)$ is the forgetful morphism and $E_H := R(p_H)_* (U_\Higgs [1] )$. Then by the axioms of triangulated categories, there is a morphism $\phi_H : T_{\Higgs_{K_X}(X)} \to E_H$ which completes the morphism between distinguished triangles, from \\ 
$\left[ q^* T_{\Bun_G(X)}[-1] \to T_{\Higgs_{K_X}(X) / \Bun_G(X)} \to T_{\Higgs_{K_X}(X)} \stackrel{+1}{\lra} \right]$ to $\left[ q^* E_B[-1] \to E_{H/B} \to E_{H} \stackrel{+1}{\lra} \right]$.
Then, by a simple direct diagram chase, we can check that $H^{0}(\phi_H)$ is an isomorphism and $H^{1}(\phi_H)$ is injective.

Over the stable locus $\Higgs^s_{K_X}(X)$, we will show that $\phi_H$ is a perfect obstruction theory. By the above argument, it is enough to show that $E_{H} = R(p_H)_* (U_\Higgs [1] )$ is perfect and its cohomology is supported in degrees $0,1$. Since $\phi_{B}$ and $\phi_{H/B}$ are perfect obstruction theories, $\phi_H$ is also perfect by the parallel argument as in \cite[Construction 3.13]{Man12}. Since the cohomologies of $E_B[-1]$ and $E_{H/B}$ are both supported in degrees $0,1,2$, the cohomologies of $E_H$ are supported in degrees $-1,0,1,2$.

Over a closed point $x=[(E_G,\vphi)] \in \Higgs_{K_X}(X)$, we have
\[
E_{H}|_{x} = R(p_x)_*( \, (\ad(E_G) \stackrel{[-,\vphi]}{\lra}  \ad(E_G) \otimes K_X)[1] \, )
\]
where $p_x : \set{x} \times X \to \set{x}$ is the projection. Therefore we have $H^{-1}(E_H|_x) \cong \cA(\cE) =  \fz(\fg) = 0$ by Theorem \ref{thm:automorphisms} when $G$ is semisimple. Moreover, we have
\[
(\ad(E_G) \stackrel{[-,\vphi]}{\lra}  \ad(E_G) \otimes K_X) \simeq ((\ad(E_G) \stackrel{[-,\vphi]}{\lra}  \ad(E_G) \otimes K_X)\dual \otimes K_X)[-1]
\]
Therefore, by Serre duality, we have
\[
H^{2}(E_H|_x) \cong H^{-1}(E_H|_x)\dual = 0.
\]
Then, parallel to the proof of \cite[Lemma 4.2]{HT10}, we conclude that $E_H$ is locally represented by a complex of vector bundles $[F^0 \to F^1]$ in degrees $0$ and $1$. 

Moreover, we have $U_{\Higgs} \cong (U_{\Higgs}[1])\dual \otimes p_2^* K_X$. 
Then the relative Grothendieck-Verdier duality tells us
\[
E_H \cong (E_H)\dual[-1]
\]
in the bounded derived category $D^b( \Higgs_{K_X}(X)^s)$. Therefore, $E_H$ is a symmetric perfect obstruction theory.

\begin{theo}
When the reductive group $G$ is semisimple and $L=K_X$, then the moduli stack of stable $K_X$-valued $G$-Higgs bundles $\Higgs_{L}(X)^s$ is a DM stack equipped with the symmetric perfect obstruction theory $\phi_{H} : T_{\Higgs_{K_X}(X)^s} \to E_H$. 
\end{theo}

\section{Appendix} \label{sect:apdx}

Here we will show that a Chern connection $D_{\fE}$ on $\fE=\ad(E_{G})$ is always induced from a connection $D_{E}$ on $E_G$.

\begin{prop}\label{prop:indconn}
Let $E_G$ be a principal $G$ bundle and let $\fE := E_G \times_G \fg$ be the adjoint bundle. Assume that a connection $D$ on $\fE$ satisfies the Leibniz rule relative to the Lie bracket. That is, for any two sections $s, t\in \Gamma(\fE)$, $D$ satisfies:
$$
D[s, t] = [D s, t] + [s, D t].
$$
Moreover, we assume that the restriction of $D$ to $\fZ:= E_G \times_G \fz(\fg)$ is $d$, that is, $D|_{\fZ} = d$. (Note that $\fZ$ is canonically isomorphic to the trivial bundle $X \times \fz(\fg)$.) $D|_{\fZ} = d$ means $D|_{\fZ}(f_1,\dots,f_r) = (df_1,\dots,df_r)$ where $r = \dim_{\CC} \fz(\fg)$ and $(f_1,\dots,f_r)$ is a (local) section of $X \times \fz(\fg)$.)

Then the connection $D$ on $\fE$ is induced from a connection on $E_G$ compatible with the holomorphic structure.
\end{prop}
\begin{proof}
By Remark \ref{rema:chernconn}, we can choose a Chern connection $D_0$ and it induces a connection $D_{\fE,0}$ on $\fE$. Thus it satisfies the above Leibniz rule by Lemma \ref{lemm:indchern}, and $D$ satisfies the Leibniz rule by assumption. 

Therefore, $\alpha := D - D_{\fE,0} \in \Gamma \left( X, \Omega^{1}_X \otimes \cE nd(\fE) \right)$ is a derivation-valued 1-form, that is, for any local sections $s, t$ of $\mathfrak{E}$, it satisfies the Leibniz rule:
$$ 
\alpha([s, t]) = [\alpha(s), t] + [s, \alpha(t)]. 
$$
Note that the space of derivations $Der(\fg)$ is isomorphic to $Der(\fg_{ss}) \oplus \End(\fz(\fg))$ where $\fg_{ss} := [\fg,\fg]$, and it is known that $Der(\fg_{ss})$ is isomorphic to the space of inner derivations. Thus we have $Der(\fg_{ss}) \cong \fg_{ss}$ and $Der(\fg) \cong Der(\fg_{ss}) \oplus \End(\fz(\fg))$. Furthermore we have
\[
Der(\fE) \cong (E_G \times_G \fg_{ss}) \oplus \End(\fZ)
\]
where $\fZ := E_G \times_G \fz(\fg)$, which is in fact isomorphic to the trivial bundle $X \times \fz(\fg)$. Therefore (local) sections of $\fZ$ are of the form $(f_1,\dots f_r)$ where $r = \dim_{\CC} \fz(\fg)$ and $f_1,\dots, f_r$ are (local) functions on $X$.

Since $D_{\fE,0}$ is a connection induced from $D_0$, by the definition of the induced connection \eqref{eq:indconn}, $D_{\fE,0}|_\fZ = d$. Moreover, $D|_\fZ = d$ by the assumption. Therefore $\alpha|_\fZ = d - d = 0$, and thus we have $\alpha \in \Gamma \left( X, \Omega^{1}_X \otimes (E_G \times_G \fg_{\ad}) \right)$.

Let $\ad : \fg \to \End(\fg)$ be the morphism sending $g \in \fg$ to $[g,-]$ and let $\fg_{\ad} := \text{Im}(\ad) \cong \fg/\fz(\fg)$. On the other hand, since $G$ is reductive, we have a decomposition
\[
\fg \cong \fz(\fg) \oplus [\fg,\fg] = \fz(\fg) \oplus \fg_{ss}.
\]
Thus we have $\fg_{\ad} \cong  \fg_{ss}$ and we have a split exact sequence: 
$$ 
0 \to \mathfrak{z}(\mathfrak{g}) \to \mathfrak{g} \xrightarrow{\ad} \mathfrak{g}_{ss} \to 0. 
$$
Since the above exact sequence splits, there is a section $s: \fg_{ss} \to \fg$ such that $\ad \circ s = \id_{\fg_{ss}}$. Then the section $s$ induces the injective morphism of vector bundles:
\[
Der(\fE) \cong \Omega_X^{1} \otimes (E_G \times_G \fg_{\ad} ) \hra \Omega_X^{1} \otimes (E_G \times_G \fg) = \Omega_X^{1} \otimes \fE. 
\]
Recall that $\alpha \in \Gamma(Der(\fE))$. Let $\wtil{\alpha}$ be the image of $\alpha$ in $\Gamma(X, \Omega_X^{1}\otimes \fE)$. Then we have $\ad(\wtil{\alpha}) = \alpha$. Therefore, if we define the new connection $D_E := D_0 + \wtil{\alpha}$, it induces the connection $D_{\fE,0} + \alpha = D$ on $\fE$. Moreover, since $\alpha$, $\wtil{\alpha}$, and $D_0$ are compatible with the holomorphic structure, so is $D_E = D_0 + \wtil{\alpha}$.

\end{proof}

\begin{coro} \label{coro:indchern}
Let $(E_G, \varphi)$ be an $L$-valued $G$-Higgs bundle. The Chern connection $D_{\mathfrak{E}}$ on $\mathfrak{E}=\ad(E_G)$ induced by the Hermitian metric $h_{\mathfrak{E}}$ satisfying the quiver vortex equation \eqref{eq:VW} is always induced from a connection $D_E$ on $E_G$.
\end{coro}
\begin{proof}
By Proposition \ref{prop:indconn}, it suffices to show that $D_\fE|_{\fZ} = d$ and $D_{\fE}$ satisfies the Leibniz rule relative to the Lie bracket.

We first show that $D_\fE|_\fZ = d$. By \eqref{eq:VW} combined with the fact that $\tau=0$ (since $\deg(\fE) = 0$), we have 
$$
i \Lambda_\omega F_\fE^{1,1} + [\vphi, \vphi^{*_h}] = 0 \in \End(\fE).
$$
Since $\vphi \in \Hom (\fE, \fE\otimes L)$, we observe that $[\vphi, \vphi^{*_h}]|_\fZ = 0$. 
\[
i \Lambda_\omega F_\fE^{1,1}|_\fZ = 0.
\]
Then, parallel to \eqref{eq:bochner}, for any $s \in H^0(X,\fZ)$, we have
$$ 
2\Delta_d |s|^2 = |D_{\frak{Z}} s|^2 - \langle i\Lambda_{\omega}F_{\frak{Z}} s, s \rangle  = |D_{\fZ} s|^2 \ge 0
$$
where $D_\fZ := D_\fE|_\fZ$, $F_{\fZ} := F_{\fE}|_{\fZ}$. Thus $|s|^2$ is a subharmonic function on the compact manifold $X$, therefore $|s|^2$ is a constant by the maximum principle. Hence $\Delta_d(|s|^2) = 0$ and therefore $|D_\fZ s|^2 = 0$, $D_\fZ s = 0$. Since $\fZ$ is isomorphic to a trivial bundle $X \times \fz(\fg)$, we have constant sections $e_1,\dots,e_r $, which form a basis on each fiber. Since $D_{\fZ}(e_1) = \dots = D_{\fZ}(e_r) = 0$, for any (local) section of the form $s = f_1 e_1 + \dots + f_r e_r$, we have $D_{\fZ}(s) = df_1 e_1 + \dots + df_r e_r$. Therefore, we also obtain $D_\fE|_\fZ = D_\fZ = d.$

Next we show that $D_{\fE}$ satisfies the Leibniz rule relative to the Lie bracket. Let $\cW := \cH om(\fE \otimes \fE, \fE)$. Let the Lie bracket $B: \fE \otimes \fE \xrightarrow{[-,-]} \fE$ be the morphism defined by the Lie bracket. Then $B$ is a holomorphic (global) section of $\cW$. The connection $D_\fE$ induces a connection $D_\cW$ in a natural way. For (local) sections $s,t$ of $\fE$, we have
$$ 
(D_{\mathcal{W}} B)(s \otimes t) = D_{\mathfrak{E}}\big( B(s \otimes t) \big) - B(D_{\mathfrak{E}}s \otimes t) - B(s \otimes D_{\mathfrak{E}}t). 
$$
By substituting $B(s \otimes t) = [s,t]$, we have
$$ 
(D_{\mathcal{W}} B)(s \otimes t) = D_{\mathfrak{E}}([s,t]) - [D_{\mathfrak{E}}s, t] - [s, D_{\mathfrak{E}}t]. 
$$
Therefore, the condition that $D_{\fE}$ satisfies the Leibniz rule relative to the Lie bracket is equivalent to $D_\cW B = 0$.

Since $B$ is a holomorphic section of $\cW$, by \eqref{eq:bochner}, we have
\begin{align} \label{eq:bochner2}
2\Delta_d (|B|^2) = |D_{\mathcal{W}} B|^2 - \langle (i\Lambda_{\omega} F^{1,1}_{\mathcal{W}}) B, B \rangle. 
\end{align}

Note that we have $i\Lambda_{\omega} F^{1,1}_{\mathfrak{E}} = -[\varphi, \varphi^{*_h}]$. Similar to the proof of Lemma \ref{lemm:endconn} (and since $F_\cW = (D_{\cW})^2$), we have
\begin{align*}
(i\Lambda_{\omega}F^{1,1}_{\cW}B)(s \otimes t) & = i\Lambda_{\omega} F^{1,1}_{\fE}(B(s \otimes t)) - B(i\Lambda_{\omega} F^{1,1}_\fE s \otimes t) - B(s \otimes i\Lambda_{\omega} F^{1,1}_{\fE} t) \\
& = i\Lambda_{\omega} F^{1,1}_\fE([s,t]) - [ i\Lambda_{\omega} F^{1,1}_\fE s, t] - [s, i\Lambda_{\omega} F^{1,1}_\fE t].
\end{align*}
We will show that the last term vanishes. Recall the equation $i \Lambda_\omega F_\fE^{1,1} + [\vphi, \vphi^{*_h}] = 0$, and recall that $\vphi : \fE \to \fE \otimes L$ is defined by an abuse of notation, $\vphi = [-, \vphi] \in \Hom(\fE, \fE \otimes L)$, which is the morphism given by the adjoint action of the Higgs field $\vphi \in H^0(X, \fE \otimes L)$. Therefore, by the Jacobi identity, we have 
$$
i\Lambda_{\omega} F^{1,1}_\fE([s,t]) - [ i\Lambda_{\omega} F^{1,1}_\fE s, t] - [s, i\Lambda_{\omega} F^{1,1}_\fE t] = [\psi,[s,t]] - [[\psi,s],t] - [s, [\psi,t]]=0
$$ 
where $\psi := -[\varphi, \varphi^{*_h}]$. Hence we obtain $i\Lambda_{\omega}F^{1,1}_{\cW}B = 0$.

Then, by \eqref{eq:bochner2}, we have $\Delta_d (|B|^2) = |D_\cW B|^2 \ge 0$. Therefore, by the maximum principle for subharmonic functions on the compact manifold $X$, $|B|^2$ is a constant so that $\Delta_d (|B|^2) = |D_\cW B|^2 = 0$, $D_\cW B = 0$.  
\end{proof}

\begin{rema}\label{rema:chernconn}
Let $X$ be a compact K\"ahler manifold and let $E_G$ be a holomorphic principal $G$-bundle, where $G$ is a reductive group. Let $K \subset G$ be a maximal compact subgroup. Then, by \cite[Section 2]{Bis05}, any smooth reduction of the structure group $E_K \subset E_G$ induces a unique connection on $E_G$ compatible with the holomorphic structure (Chern connection).

Moreover, the reduction of structure group $E_K \subset E_G$ is equivalent to a section of the fiber bundle $E_G/K \to X$. But since $G/K$ is contractible, this section always exists. In summary, we can always find a connection on $E_G$ compatible with the holomorphic structure.
\end{rema}

\begin{lemm} \label{lemm:indchern}
Let $\nabla$ be a connection on the principal $G$-bundle $E_G \to X$ and let $\nabla_{ad}$ be the induced connection on $\fE = E_G \times_G \fg$. Then $\nabla_{ad}$ satisfies the Leibniz rule. 
\end{lemm}
\begin{proof}
Note that the connection $\nabla$ is represented by a connection 1-form $\omega \in \Omega^1 (E_G,\fg)$. We also note that every (local) section $s \in \Gamma(U,\fE)$ for an open subset $U \subset X$ corresponding to a smooth function $f_s : E_G \to \mathfrak{g}$ satisfies the $G$-equivariance condition:
$$ 
f_s(p \cdot g) = \text{Ad}(g^{-1}) \big( f_s(p) \big) 
$$
for any $p \in E_G|_U$ and $g \in G$ (\cite[Chapter II, Example 5.2]{KN63}). Then the induced connection $\nabla_{ad}$ on $\fE$ is defined by
\begin{align}\label{eq:indconn}
\nabla_{ad}(f_s) := d f_s + [\omega,f_s].
\end{align}
(See \cite[Lemma in page 115]{KN63} and \cite[Chapter II, Lemma 1 (1)]{KN63} for details.) Note that in fact $\nabla_{ad}(f_s) \in \Omega^1(E_G|_U,\fg)$ is horizontal, hence it becomes an element of $\Gamma(X, \Omega_X^1 \otimes \fE)$.

Now let $f_1, f_2$ be $G$-equivariant functions corresponding to the local sections $s_1$ and $s_2$ of $\fE$. Then the section $[s_1, s_2]$ corresponds to the pointwise Lie bracket of the functions $[f_1, f_2]$. Then we have
$$ 
\nabla_{ad}([f_1, f_2]) = d([f_1, f_2]) + [\omega, [f_1, f_2]]. 
$$
Then, by the product rule for differentiation, we have
$$ 
d([f_1, f_2]) = [d f_1, f_2] + [f_1, d f_2].
$$
Moreover, by the Jacobi identity, we have:
$$ 
[\omega, [f_1, f_2]] = [[\omega, f_1], f_2] + [f_1, [\omega, f_2]].
$$
Therefore, the induced connection $\nabla_{ad} = d + [\omega,-]$ also satisfies the Leibniz rule.

\end{proof}


\begin{thebibliography}{9}

\bibitem{Alp26}
J. Alper, \textit{Stacks and Moduli}, 2026.

\bibitem{ACGP03}
L. Álvarez-Cónsul and O. García-Prada,
\textit{Hitchin-Kobayashi correspondence, quivers, and vortices},
Commun. Math. Phys. \textbf{238} (2003), 1--33.

\bibitem{ACGP03b}
L. Álvarez-Cónsul and O. García-Prada,
\textit{Hitchin-Kobayashi correspondence, quivers, and vortices},
arXiv preprint math/0112161 (2001).

\bibitem{AB01}
B. Anchouche, I. Biswas,
\textit{Einstein-Hermitian connections on polystable principal bundles over a compact K{\"a}hler manifold},
American Journal of Mathematics. \textbf{123(2)} (2001), 207--228.

\bibitem{BS09}
I. Biswas and G. Schumacher,
\textit{Yang-Mills equation for stable Higgs sheaves},
Int. J. Math. \textbf{20} (2009), 541--556.

\bibitem{Bis05}
I. Biswas, \textit{Stable bundles and extension of structure group},
Differential Geometry and its Applications.
\textbf{23(1)} (2005), 67--78.

\bibitem{Bis09}
I. Biswas,
\textit{On the stable principal Higgs sheaves},
Differential Geom. Appl. \textbf{27} (2009), 344--351.

\bibitem{CL12}
H. -L. Chang and J. Li,
\textit{Gromov--Witten invariants of stable maps with fields},
International mathematics research notices.
\textbf{2012(18)} (2012), 4163--4217.



\bibitem{Falt93}
G. Faltings, Stable G-bundles and projective connections, J. Alg. Geom. \textbf{2(3)} (1993), 507--568.

\bibitem{HR14}
J. Hall, D. Rydh,
\textit{The Hilbert Stack}, Advances in Mathematics. \textbf{253} (2014), 194–-233.

\bibitem{Hit87}
N. J. Hitchin,
\textit{The self-duality equations on a Riemann surface},
Proceedings of the London Mathematical Society.
\textbf{3(1)} (1987), 59--126.

\bibitem{HT10}
D. Huybrechts, R. P. Thomas,
\textit{Deformation-obstruction theory for complexes via Atiyah and Kodaira–Spencer classes}, Mathematische Annalen.
\textbf{346(3)} (2010), 545--569.

\bibitem{KMMP25}
D. Kern et al, 
\textit{Derived moduli of sections and push-forwards},
Selecta Mathematica.
\textbf{31(2)} (2025), 40.

\bibitem{Ko87}
S. Kobayashi, \textit{Differential Geometry of Complex Vector Bundles}, Publications of the Math. Society of Japan, vol. 15, Iwanami Shoten Publishers and Princeton University Press, (1987).

\bibitem{KN63}
S. Kobayashi, K. Nomizu, \textit{Foundations of Differential Geometry, Volume 1}, Interscience Publishers, (1963).

\bibitem{Man12}
C. Manolache, \textit{Virtual pull-backs},
Journal of Algebraic Geometry.
\textbf{21(2)} (2012), 201--245.


\bibitem{Ram75}
A. Ramanathan,
\textit{Stable principal bundles on a compact Riemann surface},
Math. Ann. \textbf{213} (1975), 129--152.

\bibitem{Sch25}
S. Schirren, \textit{A virtual structure for symplectic Higgs bundles}, arXiv preprint arXiv:2510.24531 (2025).

\bibitem{Sim88}
C. T. Simpson,
\textit{Constructing variations of Hodge structure using Yang-Mills theory and applications to uniformization},
Journal of the American Mathematical Society.
(1988), 867--918.

\bibitem{Sim94}
C. T. Simpson, \textit{Moduli of representations of the fundamental group of a smooth projective variety I}, Publications Math{\'e}matiques de l'IH{\'E}S. \textbf{79} (1994), 47--129.

\bibitem{Tan14}
Y. Tanaka,
\textit{Stable sheaves with twisted sections and the Vafa-Witten equations on smooth projective surfaces},
Manuscripta Math. \textbf{146} (2015), 351--358.

\bibitem{TT19}
Y. Tanaka and R. P. Thomas, \textit{Vafa-Witten invariants for projective surfaces I: stable case}, 
Journal of Algebraic Geometry.
\textbf{29(4)} (2019).

\end{thebibliography}
\end{document}